\documentclass[12pt,reqno]{amsart}
\usepackage{indentfirst, amssymb, amsmath, amsthm, mathrsfs, setspace, indentfirst, enumerate,  mathrsfs, amsmath, amsthm}
\usepackage[bookmarksnumbered, colorlinks, plainpages]{hyperref}
\usepackage{mathrsfs}
\usepackage{cite}

\newtheorem*{theoA}{Theorem A}
\newtheorem*{theoB}{Theorem B}
\newtheorem*{theoC}{Theorem C}
\newtheorem*{theoD}{Theorem D}
\newtheorem*{theoE}{Theorem E}

\newtheorem*{cor A}{Corollary A}
\newtheorem*{cor B}{Corollary B}
\newtheorem{propo}{Proposition}[section]

\newtheorem{theo}{Theorem}[section]
\newtheorem{lem}{Lemma}[section]
\newtheorem{cor}{Corollary}[section]

\newtheorem{defi}{Definition}[section]
\newtheorem{rem}{Remark}[section]

\newcommand{\ol}{\overline}
\newcommand{\be}{\begin{equation}}
\newcommand{\ee}{\end{equation}}
\newcommand{\beas}{\begin{eqnarray*}}
\newcommand{\eeas}{\end{eqnarray*}}
\newcommand{\bea}{\begin{eqnarray}}
\newcommand{\eea}{\end{eqnarray}}

\numberwithin{equation}{section}
\begin{document}
\title[Univalence and Coefficient Bounds for Pluriharmonic Mappings]{On Stable Univalence and Coefficient Estimates for a Class of Pluriharmonic Mappings in Convex Reinhardt Domains}
\date{}
\author[M. B. A\MakeLowercase{hamed}, S. M\MakeLowercase{ajumder} \MakeLowercase{and} D. P\MakeLowercase{ramanik}]{M\MakeLowercase{olla} B\MakeLowercase{asir} A\MakeLowercase{hamed}$^*$, S\MakeLowercase{ujoy} M\MakeLowercase{ajumder} \MakeLowercase{and} D\MakeLowercase{ebabrata} P\MakeLowercase{ramanik}}

\address{Molla Basir Ahamed,
	Department of Mathematics,
	Jadavpur University,
	Kolkata-700032, West Bengal, India.}
\email{mbahamed.math@jadavpuruniversity.in}

\address{Sujoy Majumder, Department of Mathematics, Raiganj University, Raiganj, West Bengal-733134, India.}
\email{sm05math@gmail.com}

\address{Debabrata Pramanik, Department of Mathematics, Raiganj University, Raiganj, West Bengal-733134, India.}
\email{debumath07@gmail.com}

\renewcommand{\thefootnote}{}
\footnote{*\emph{Corresponding Author}: Molla Basir Ahamed.}

\renewcommand{\thefootnote}{\arabic{footnote}}
\setcounter{footnote}{0}

\begin{abstract}
In this paper, we investigate the geometric properties of complex-valued pluriharmonic mappings defined over convex Reinhardt domains in $\mathbb{C}^n$. We first establish a multidimensional analogue of the Noshiro-Warschawski Theorem, providing sufficient conditions for the univalence of pluriharmonic mappings based on the real part of their partial derivatives. Furthermore, we introduce and study the class $\mathcal{B}_{\mathcal{H}_{n}^{0}}(M)$ of normalized pluriharmonic mappings, characterized by a specific bound on the sum of their second-order partial derivatives. We prove a one-to-one correspondence between this pluriharmonic class and a corresponding class of holomorphic functions, extending known results from the planar harmonic case to higher dimensions. Specifically, we show that a pluriharmonic mapping $f=h+\overline{g}$ is stable pluriharmonic univalent if and only if its holomorphic counterpart $F=h+g$ is stable holomorphic univalent on the unit polydisk $\mathbb{P}\Delta(0;1)$. Finally, we provide sharp coefficient estimates and sufficient conditions for functions to belong to the class $\mathcal{B}_{\mathcal{H}_{n}^{0}}(M)$. Our results generalize several classical theorems in the theory of univalent harmonic functions to the setting of several complex variables.
\end{abstract}
\thanks{Typeset by \AmS -\LaTeX}
\maketitle

\section{\bf Introduction and Preliminaries}
For $z=(z_1,\ldots,z_n)$ and $w=(w_1,\ldots,w_n)$ in $\mathbb{C}^{n}$, we denote $\langle z,w\rangle=z_1\ol w_1+\ldots+z_n \ol w_n$ and $||z||=\sqrt{\langle z,z\rangle}$. The absolute value of a complex number $z_1$ is denoted by $|z_1|$ and for $z\in\mathbb{C}^n$, we define
\begin{align*}
	||z||_{\infty}=\max\limits_{1\leq i\leq n}|z_i|.
\end{align*}

Throughout the paper, for $z=(z_1,\ldots,z_n)\in\mathbb{C}^n$ and $w=(w_1,\ldots,w_n)\in\mathbb{C}^{n}$, $p\in\mathbb{N}$ and $\lambda\in\mathbb{C}$, we define $z+w=(z_1+w_1,\ldots,z_n+w_n)$, $z.w=(z_1w_1,\ldots,z_n w_n)$, $\lambda z=(\lambda z_1,\ldots,\lambda z_n)$, $z^p=(z_1^p,\ldots,z_n^p)$ and $\ol z=(\ol z_1,\ldots,\ol z_n)$.

\smallskip
If $x=(x_1,\ldots,x_n)$ and $y=(y_l,\ldots,y_n)$ are vectors in $\mathbb{R}^{n}$, then we define $x\geq y$ if $x_j\geq y_j$ for all $j$ and $x>y$ if $x_j>y_j$ for all $j$.
A multi-index $\alpha=(\alpha_1,\ldots,\alpha_n)$ of dimension $n$ consists of n non-negative integers $\alpha_j,\;1\leq j\leq n$; the degree of a multi-index $\alpha$ is the sum $|\alpha|=\sum_{j=1}^n \alpha_j$ and we denote $\alpha!=\alpha_1!\ldots \alpha_n!$. For $z=(z_1,\ldots,z_n)\in\mathbb{C}^n$ and a multi-index $\alpha=(\alpha_1,\ldots,\alpha_n)$, we define 
\[z^{\alpha}=\prod\limits_{j=1}^n z_j^{\alpha_j}\;\;\text{and}\;\;|z|^{\alpha}=\prod\limits_{j=1}^n |z_j|^{\alpha_j}.\]

\begin{defi} An open set $\Omega\subset \mathbb{C}^n$ is called a \emph{Reinhardt} domain if $(z_1,z_2,\ldots,z_n)\in\Omega$ implies $\left(e^{i\theta_1}z_1,e^{i\theta_2}z_2,\ldots,e^{i\theta_n}z_n\right)\in \Omega$ for arbitrary real $\theta_1,\theta_2,\ldots,\theta_n$.
\end{defi}

\begin{defi} A set $\Omega\subset\mathbb{C}^n$ is said to be \emph{starlike} with respect to a point $z_0$ if $(1-t)z_0+tz\in \Omega$ for all $z\in \Omega$ and for all $t\in[0,1]$. The set $\Omega$ is said to be \emph{convex} if it is starlike with respect to each of its points, i.e., if $(1-t)z+tw\in\Omega$ for all $z,w\in\Omega$ and for all $t\in[0,1]$. $\Omega$ is said to be \emph{close-to-convex} if the complement of $\Omega$ can be written as union of non-intersecting half lines.
\end{defi}

\begin{defi}
An open polydisk (or open polycylinder) in $\mathbb{C}^n$ is a subset $\mathbb{P}\Delta(a;r)\subset \mathbb{C}^n$ of the form 
\[\mathbb{P}\Delta(a;r)=\prod\limits_{j=1}^n \Delta(a_j;r_j)=\lbrace z\in\mathbb{C}^n: |z_i-a_i|<r_i,\;i=1,2,\ldots,n\rbrace,\]
the point $a=(a_1,\ldots,a_n)\in\mathbb{C}^n$ is called the centre of the polydisk and $r=(r_1,\ldots,r_n)\in\mathbb{R}^n\;(r_i>0)$ is called the polyradius. It is easy to see that
\begin{align*}
	\mathbb{P}\Delta(0;1)=\mathbb{P}\Delta(0;1_n)=\prod\limits_{j=1}^n \Delta(0_n;1_n).
\end{align*}
\end{defi}

The closure of $\mathbb{P}\Delta(a;r)$ will be called the closed polydisk with centre $a$ and polyradius $r$ and will be denoted by $\ol{\mathbb{P}\Delta}(a;r)$. We denote by $C_k(a_k;r_k)$ the boundary of $\Delta(a_k;r_k)$, \textit{i.e.,} the circle of radius $r_k$ with centre $a_k$ on the $z_k$-plane. Of course $C_k(a_k,r_k)$ is represented by the usual parametrization 
\begin{align*}
	\theta_k\to \gamma(\theta_k)=a_k+r_ke^{i\theta_k},\; \mbox{where}, 0\leq \theta_k\leq 2\pi.
\end{align*} The product $C^n(a;r):=C_1(a_1;r_1)\times\ldots\times C_n(a_n;r_n)$ is called the determining set of the polydisk $\mathbb{P}\Delta(a;r)$.

\begin{rem} It is easy to conclude that an open polydisk $\mathbb{P} \Delta(a;r)\subset \mathbb{C}^n$ is a convex Reinhardt domain.
\end{rem} 

For $a=(a_1,\ldots,a_n)\in\mathbb{C}^{n}$ and $r>0$, we define
\[\mathbb{B}_n(a;r)=\{z\in\mathbb{C}^n: ||z-a||<r\}\;\;\text{and}\;\;\mathbb{B}_n(r)=\mathbb{B}_n(0,r).\]

Clearly, $\mathbb{B}_n=\mathbb{B}_n(1)$ is the unit ball in the complex space $\mathbb{C}^n$ of dimension $n$. The unit disk in
the complex plane is denoted by $\mathbb{D}$.\vspace{1.2mm}

A complex-valued function $f$ on a subset $\Omega\subset \mathbb{C}^n$ is a mapping from $\Omega$ into the complex plane. The value of the function $f$ at a point $z=(z_1,\ldots,z_n)\in \Omega$ is denoted by $f(z)=f(z_1,\ldots,z_n)$.
\begin{defi} If $f(z)=f(z_1,\ldots,z_n)$ is continuous in $\Omega\subset \mathbb{C}^n$ and holomorphic in each variable $z_k, k=1,\ldots,n$, separately, then $f(z)$ is said to be holomorphic in $\Omega$. We also call $f(z)=f(z_1,\ldots,z_n)$ a holomorphic function of $n$ variables $z_1,\ldots,z_n$.
\end{defi}
Here, by saying that $f(z_1,\ldots,z_k,\ldots,z_n)$ is holomorphic in $z_k$, separately, we mean that the function $f(z_1,\ldots,z_n)$ is a holomorphic function in $z_k$ when the other variables $z_1,\ldots,z_{k-1},z_{k+1},\ldots,z_n$ are fixed. We write
\begin{align*}
z_j = x_j + i y_j \qquad (i^2 = -1,\; j = 1,\dots,n),
\end{align*}
where \(x_j\) and \(y_j\) are real numbers. We set
\begin{align*}
f(z) = u(x,y) + i v(x,y),
\end{align*}
where $u(x,y)$ and $v(x,y)$ are the real and imaginary parts of $f(z)$; $x = (x_1,\dots,x_n)$ and $y = (y_1,\dots,y_n)$.
The Cauchy--Riemann equations for each $z_j\;(j=1,\dots,n)$ are
\begin{align}\label{Eq 1.1}
\frac{\partial u}{\partial x_j}
=
\frac{\partial v}{\partial y_j}\quad \text{and}\quad
\frac{\partial u}{\partial y_j}
=
-\,\frac{\partial v}{\partial x_j}
\qquad (j=1,\dots,n).
\end{align}

By differentiating (\ref{Eq 1.1}) with respect to $x_k$ and $y_k$, we see that both $u$ and $v$ satisfy the
following system of partial differential equations of second order:
\begin{align}\label{Eq 1.2}
\frac{\partial^2}{\partial x_j \partial x_k}
+
\frac{\partial^2}{\partial y_j \partial y_k}
= 0
\quad \text{and} \quad
\frac{\partial^2}{\partial x_j \partial y_k}
-
\frac{\partial^2}{\partial x_k \partial y_j}
= 0
\qquad (j,k = 1,\dots,n).
\end{align}

For a complex variable $z_j = x_j + i y_j$, we define
\begin{align}\label{Eq 1.3}
\frac{\partial}{\partial z_j}
= \frac{1}{2}\left( \frac{\partial}{\partial x_j}
- i \frac{\partial}{\partial y_j} \right),\quad
\frac{\partial}{\partial \bar z_j}
= \frac{1}{2}\left( \frac{\partial}{\partial x_j}
+ i \frac{\partial}{\partial y_j} \right).
\end{align}
\begin{align*}
\partial =\sum\limits_{j=1}^n \frac{\partial }{\partial z_j} d z_j\quad \text{and} \quad \ol{\partial} =\sum\limits_{j=1}^n \frac{\partial }{\partial z_j} d \ol {z}_j.
\end{align*}

Here, $dx_j$, $dy_j$ etc. are symbols of linearly independent vectors. Then condition (\ref{Eq 1.1}), that the complex-valued function $f(z_1,\dots,z_n)$ is differentiable with respect to the variable $z_j$,
becomes
\begin{align*}
\frac{\partial f(z)}{\partial \bar z_j} = 0,
\qquad j = 1,\dots,n.
\end{align*}

Let $f(z) = u(x,y) + i v(x,y)\;(x = (x_1,\dots,x_n), y = (y_1,\dots,y_n))$, where both $u$ and $v$ are of $C^2$-class. A direct calculation on (\ref{Eq 1.3}) shows that 
\begin{align}\label{Eq 1.5}
4\frac{\partial^2 f(z)}{\partial z_j \partial \bar z_k}=&\frac{\partial^2 u(x,y)}{\partial x_j x_k}+\frac{\partial^2 u(x,y)}{\partial y_j y_k}+i\left(\frac{\partial^2 u(x,y)}{\partial x_j x_k}+\frac{\partial^2 u(x,y)}{\partial y_j y_k}\right)\\&+i\left(\frac{\partial^2 u(x,y)}{\partial x_j y_k}-\frac{\partial^2 u(x,y)}{\partial x_k y_j}\right)-\left(\frac{\partial^2 v(x,y)}{\partial x_j y_k}-\frac{\partial^2 v(x,y)}{\partial x_k y_j}\right).\nonumber
\end{align}

We know that a function $f$ defined on an open subset $U\subset \mathbb{R}^n$ is said to be of $C^k$-class if $f$ is $k$-times continuously differentiable.

\begin{defi} A function $f:\Omega \to \mathbb{C}$ defined on an open set $\Omega\subset \mathbb{C}^n$ is said to be
holomorphic if $f$ is of $C^1$-class and satisfies
\begin{align}\label{Eq 1.4}
\ol{\partial} f=0,\quad \text{that is},\quad \frac{\partial f(z)}{\partial \bar z_j}=0
\end{align}
on $\Omega$ for all $j$.
\end{defi}

\begin{propo}\label{Pro:1} Suppose $f(z)$ is a holomorphic function in a convex Reinhardt domain $\Omega\subset \mathbb{C}^n$ such that $\frac{\partial f(z_0)}{\partial z_j}\neq 0$ for all $j=1,2,\ldots,n$, where $z_0\in\Omega$. Then $f(z)$ is univalent in some convex Reinhardt neighborhood of $z_0$.
\end{propo}
\begin{rem}
This is a valid multidimensional extension of the local inverse mapping theorem logic. In one variable, $f'(z_0) \neq 0$ implies local univalence; in $n$ variables, the condition that all partial derivatives are non-zero at a point, combined with the continuity of those derivatives, ensures a neighborhood exists where the function remains injective (univalent).
\end{rem}
\begin{proof} Since $\frac{\partial f(z_0)}{\partial z_j}\neq 0$ and $\frac{\partial f(z)}{\partial z_j}$ is continuous at $z_0$ for all $j=1,2,\ldots,n$, there exists a convex Reinhardt neighborhood $N(z_0)$ such that
\begin{align}\label{Pro1.1}
\left|\frac{\partial f(z)}{\partial z_j}-\frac{\partial f(z_0)}{\partial z_j}\right|<\frac{1}{2}\left|\frac{\partial f(z_0)}{\partial z_j}\right|
\end{align}
for all $z\in N(z_0)$ and for all $j=1,2,\ldots,n$. Suppose $\frac{\partial f(z_0)}{\partial z_j}=r_je^{i\theta_j}$ for $j=1,2,\ldots,n$.
Let $z=(z_1,z_2,\ldots,z_n)\in N(z_0)$ and $w=(w_1,w_2,\ldots,w_n) \in N(z_0)$ such that $z \neq w$. Since $z\neq w$, for the sake of simplicity we may assume that $z_j\neq w_j$ for all $j=1,\ldots,n$. Let $w_j-z_j=r_j'e^{i\alpha_j}$, $j=1,\ldots,n$.
Since $\Omega$ is a Reinhardt domain, it follows that 
\begin{align*}
	\tilde z=\left(e^{-i(\alpha_1+\theta_1)}z_1,e^{-i(\alpha_2+\theta_2)}z_2,\ldots,e^{-i(\alpha_n+\theta_n)}z_n\right)\in N(z_0)
\end{align*} and 
\begin{align*}
	\tilde w=\left(e^{-i(\alpha_1+\theta_1)}w_1,e^{-i(\alpha_2+\theta_2)}w_2,\ldots,e^{-i(\alpha_n+\theta_n)}w_n\right)\in N(z_0).
\end{align*} Clearly, $\tilde z\neq \tilde w$. Then, because $N(z_0)$ is a convex domain, 
$\xi(t)=(1 - t) \tilde z + t\tilde w \in N(z_0)$ for $0 \le t \le 1$. Let $\varphi: [0,1]\to \mathbb{C}$ be defined by 
\begin{align*}
\varphi(t)=f(\xi(t))-\sum\limits_{j=1}^n \frac{\partial f(z_0)}{\partial z_j}. ((1-t)z_j+tw_j)e^{-i(\alpha_j+\theta_j)}.
\end{align*}
It is easy to see that
\begin{align*}
\varphi(1)=f(\tilde w)-\sum\limits_{j=1}^n \frac{\partial f(z_0)}{\partial z_j}.w_je^{-i(\alpha_j+\theta_j)}\quad \text{and}\quad \varphi(0)=f(\tilde z)-\sum\limits_{j=1}^n \frac{\partial f(z_0)}{\partial z_j}.z_je^{-i(\alpha_j+\theta_j)}.
\end{align*}

Note that
\begin{align*}
\varphi'(t)=&\sum\limits_{j=1}^n\frac{\partial f(\xi(t))}{\partial z_j}.(w_j-z_j)e^{-i(\alpha_j+\theta_j)}-\sum\limits_{j=1}^n \frac{\partial f(z_0)}{\partial z_j}.(w_j-z_j)e^{-i(\alpha_j+\theta_j)}\\=&\sum\limits_{j=1}^n\left( \frac{\partial f(\xi(t))}{\partial z_j}-\frac{\partial f(z_0)}{\partial z_j}\right)\cdot(w_j-z_j)e^{-i(\alpha_j+\theta_j)}.
\end{align*}
Since
\begin{align*}
\varphi(1)-\varphi(0)=\int_{0}^1 \frac{d \varphi(t)}{dt}dt,
\end{align*}
we have
\begin{align}\label{Pro1.2}
&f(\tilde w)-f(\tilde z)-\sum\limits_{j=1}^n \frac{\partial f(z_0)}{\partial z_j}\cdot(w_j-z_j)e^{-i(\alpha_j+\theta_j)}\\=&
\int\limits_{0}^1\left( \sum\limits_{j=1}^n\left(\frac{\partial f(\xi(t))}{\partial z_j}-\frac{\partial f(z_0)}{\partial z_j}\right)\cdot(w_j-z_j)e^{-i(\alpha_j+\theta_j)}\right) dt.\nonumber
\end{align}
Using (\ref{Pro1.1}) to (\ref{Pro1.2}), we obtain
\begin{align}\label{Pro1.3}
&\left|f(\tilde w)-f(\tilde z)-\sum\limits_{j=1}^n \left|\frac{\partial f(z_0)}{\partial z_j}\right|.|w_j-z_j|\right|\\\leq&
\int\limits_{0}^1\left( \sum\limits_{j=1}^n\left|\frac{\partial f(\xi(t))}{\partial z_j}-\frac{\partial f(z_0)}{\partial z_j}\right|\cdot|w_j-z_j|\right) dt\nonumber\\<&
\frac{1}{2}\sum\limits_{j=1}^n\left|\frac{\partial f(z_0)}{\partial z_j}\right|\cdot|w_j-z_j|.\nonumber
\end{align}
By the given condition, we have
\begin{align*}
\sum\limits_{j=1}^n\left|\frac{\partial f(z_0)}{\partial z_j}\right|.|w_j-z_j|>0.
\end{align*}
Thus, by the triangle inequality, we obtain from (\ref{Pro1.3}) that
\begin{align*}
|f(\tilde w)-f(\tilde z)|>\frac{1}{2}\sum\limits_{j=1}^n\left|\frac{\partial f(z_0)}{\partial z_j}\right|.|w_j-z_j|>0.
\end{align*}
Hence, $f$ is univalent in the convex Reinhardt neighborhood $N(z_0)$.
\end{proof}
\begin{rem}
	The proof of Proposition 1.1 utilizes the properties of convexity and Reinhardt domains to bridge the gap between local derivative properties and univalence within a neighborhood. This approach remains consistent with the generalized results for pluriharmonic mappings presented in Section 2.
\end{rem}
Let $f(z)$ be a holomorphic function in a domain $\Omega\subset \mathbb{C}^n$, and $c=(c_1,\ldots,c_n)\in \Omega$. Then in a polydisk $\mathbb{P}\Delta(c;r)\subset \Omega$ with centre $c$, $f(z)$ has a power series expansion in $z_1-c_1,\ldots,z_n-c_n$,
\begin{align*}
	f(z)&=\sum\limits_{\alpha_1,\alpha_2,\ldots,\alpha_n=0}^{\infty} a_{\alpha_1,\alpha_2,\ldots,\alpha_n}(z_1-c_1)^{\alpha_1}(z_2-c_2)^{\alpha_2}\ldots (z_n-c_n)^{\alpha_n}\\&=
	\sum\limits_{|\alpha|=0} a_{\alpha}(z-c)^{\alpha}=\sum\limits_{|\alpha|=0}^{\infty} P_{|\alpha|}(z-c),
\end{align*}
which is absolutely convergent in $\mathbb{P}\Delta(c;r)$, where the term $P_k(z-c)$ is a homogeneous polynomial of degree $k$. 
\vspace{1.2mm}
Let $G\not=\varnothing$ be an open subset of $\mathbb{C}^n$. Let $f$ be a holomorphic function on $G$. For a point $a\in\mathbb{C}^n$, we write $f(z)=\sum_{i=0}^{\infty}P_i(z-a)$, where the term $P_i(z-a)$ is either identically zero or a homogeneous polynomial of degree $i$. Denote the zero-multiplicity of $f$ at $a$ by $k=\min\{i:P_i(z-a)\not\equiv 0\}$. Clearly $1$ is the zero-multiplicity of $f$ at $a$ when $f(a)=0$ and $\frac{\partial f(a)}{\partial z_j}\neq 0$ for some $j=1,2,\ldots,n$.\vspace{1.2mm}

\begin{defi}
A real-valued function $\phi(x,y)$, where $x = (x_1,\dots,x_n)$ and $y = (y_1,\dots,y_n)$ is \emph{pluriharmonic} if it satisfies the conditions  (\ref{Eq 1.2}) is called \emph{pluriharmonic}. Thus a continuous complex-valued function $f(z)=u(x,y)+iv(x,y)$, where
$x = (x_1,\dots,x_n)$ and $y = (y_1,\dots,y_n)$ is a complex-valued \emph{pluriharmonic} function in a domain $\Omega\subset \mathbb{C}^n$, if both $u(x,y)$ and $v(x,y)$ are real-valued \emph{pluriharmonic} functions in $\Omega$.
If $u(x,y)$ and $v(x,y)$, where $x = (x_1,\dots,x_n)$ and $y = (y_1,\dots,y_n)$ satisfy (\ref{Eq 1.1}), then we call $v(x,y)$ a \emph{pluriharmonic conjugate} of $u(x,y)$.
\end{defi}

 Thus for functions $f(z)=u(x,y)+iv(x,y)$, where
$x = (x_1,\dots,x_n)$ and $y = (y_1,\dots,y_n)$ with continuous second order partial derivatives, it is clear from (\ref{Eq 1.4}) and (\ref{Eq 1.5}) that $\frac{\partial f(z)}{\partial z_j}$ is holomorphic on $\Omega$ for all $j$ if  $f(z)$ is \emph{pluriharmonic} function.

\smallskip
\begin{propo}
In a simply connected domain $\Omega\subset \mathbb{C}^n$, a complex-valued \emph{pluriharmonic} function $f(z)$ has the representation $f=h+\bar g$, where $h$ and $g$ are holomorphic in $\Omega$.
\end{propo}
\begin{proof} We recall that if $f$ is a pluriharmonic mapping on $\Omega$, then $\frac{\partial f}{\partial z_j}$ is holomorphic in $\Omega$ for each $j=1, 2, \dots, n$. Let $h$ be a holomorphic function in $\Omega$ such that 
	\begin{align*}
		\frac{\partial h}{\partial z_j} = \frac{\partial f}{\partial z_j}\; \mbox{for}\; j=1, 2, \ldots, n.
	\end{align*}  Defining $g = \overline{f} - \overline{h}$, we observe that$$ \frac{\partial g}{\partial \overline{z}_j} = \overline{\frac{\partial f}{\partial z_j}} - \overline{\frac{\partial h}{\partial z_j}} = 0 \quad \text{in } \Omega, \quad j=1, 2, \dots, n, $$by the construction of $h$. This implies that $g$ is holomorphic in $\Omega$. Consequently, the pluriharmonic function $f$ admits the canonical representation $f = h + \overline{g}$, where both $h$ and $g$ are holomorphic in $\Omega$.
\end{proof}
\begin{rem}
	The significance of Proposition 1.2 lies in its role as a structural foundation for extending the theory of univalent harmonic mappings from the complex plane to $\mathbb{C}^n$. By establishing the $f = h + \overline{g}$ representation in simply connected domains, we provide a rigorous pathway to generalize the classical results of Clunie and Sheil-Small [2] and the stable univalence criteria of Hernandez and Martin [7] to the setting of several complex variables. This decomposition is instrumental in the proofs of our main results, including the multidimensional Noshiro-Warschawski Theorem (Theorem 2.1) and the stable univalence equivalence (Theorem 2.2).
\end{rem}

\medskip
Now we recall a result of univalent harmonic function on $\Omega\subset \mathbb{C}$ due to Ponnusamy et al. \cite{Ponnusamy-Yamamoto-Yanagihara-CVEE-2013}.

\begin{theoA}\cite[Lemma 1.1]{Ponnusamy-Yamamoto-Yanagihara-CVEE-2013} Suppose $f = h + \overline{g}$ is a harmonic mapping in a convex domain $\Omega\subset \mathbb{C}$ such that
\begin{align*}
\operatorname{Re}\!\left(e^{i\gamma} h'(z)\right) > |g'(z)| \quad \text{for all } z \in \Omega,
\end{align*}
and for some $\gamma \in \mathbb{R}$. Then $f$ is univalent in $\Omega$.
\end{theoA}
Let $\mathcal{H}_n$ denote the class of complex-valued pluriharmonic functions $f$ in $\mathbb{P} \Delta(0;1)$, normalized by 
\begin{align*}
	f(0)= 0\; \mbox{and}\;\left(\frac{\partial f(0)}{\partial z_1},\ldots,\frac{\partial f(0)}{\partial z_n}\right)=(1,1,\ldots,1).
\end{align*} Obviously $\mathcal{H}=\mathcal{H}_1$ is the the class of complex-valued harmonic functions $f$ in the unit disk $\mathbb{D}=\{z\in\mathbb{C}: |z|<1\}$, normalized by $f(0)=0=\frac{\partial f(0)}{\partial z}-1$.
Each function $f$ in $\mathcal{H}$ can be expressed as $f=h+\ol g$, where $h$ and $g$ are holomorphic functions in $\mathbb{P} \Delta(0;1)$. Here $h$ and $g$ are called the holomorphic and co-holomorphic parts of $f$ respectively, and have power series representations
\begin{align}\label{Eq 1.6}
h(z)=\sum\limits_{k=1}^n z_j+\sum\limits_{k=2}^{\infty}\sum\limits_{|\alpha|=k} a_{\alpha} z^{\alpha}\quad {and}\quad g(z)=\sum\limits_{k=1}^{\infty}\sum\limits_{|\alpha|=k} b_{\alpha} z^{\alpha}.
\end{align}

Let $\mathcal{S}_{\mathcal{H}_n}$ be the subclass of $\mathcal{H}_n$ consisting of univalent (i.e. one-to-one) pluriharmonic mappings on $\mathbb{P} \Delta(0;1)$. Let
\begin{align*}
\mathcal{H}^0_n=\left\lbrace f\in\mathcal{H}_n: \ol{\partial} f(0)=0\right\rbrace
\end{align*}
and 
\begin{align*}
\mathcal{S}^0_{\mathcal{H}_n}=\left\lbrace f\in\mathcal{S}_{\mathcal{H}_n}: \ol{\partial} f(0)=0\right\rbrace.
\end{align*}

Hence for any function $f=h+\ol g$ in $\mathcal{H}^0_n$, its holomorphic and co-holomorphic parts can be represented by
\begin{align}\label{Eq 1.7}
h(z)=\sum\limits_{k=1}^n z_j+\sum\limits_{k=2}^{\infty}\sum\limits_{|\alpha|=k} a_{\alpha} z^{\alpha}\quad {and} \quad g(z)=\sum\limits_{k=2}^{\infty}\sum\limits_{|\alpha|=k} b_{\alpha} z^{\alpha}
\end{align}
respectively. We note that $\mathcal{S}_{\mathcal{H}_n}$ reduces to $\mathcal{S}$, the class of normalized univalent holomorphic functions in $\mathbb{P} \Delta(0;1)$ whenever the co-holomorphic part of $f$ is zero, i.e., $g(z)\equiv 0$ in $\mathbb{P} \Delta(0;1)$.
In 1984, Clunie and Sheil-Small \cite{Clunie-Sheil-Small-1984} investigated the class $\mathcal{S}_{\mathcal{H}}$, together with some geometric subclasses. Subsequently, the class $\mathcal{S}_{\mathcal{H}}$ and its subclasses have been extensively studied by several authors (see \cite{Ahamed-Allu-BMMSS-2022}-\cite{Bshouty-Joshi-Joshi-2013}, \cite{Clunie-Sheil-Small-1984,Kalaj-Ponnusamy-Vuorinen-2014, Hernandez-Martin-2013}).
\begin{defi}
	A complex-valued pluriharmonic mapping $f\in \mathcal{H}_n$ is said to be starlike if $f(\mathbb{P} \Delta(0;1))$ is a starlike
	domain with respect to the origin. We denote the class of pluriharmonic starlike functions in $\mathbb{P} \Delta(0;1)$ by $\mathcal{S}^*_{\mathcal{H}_n}$. A function $f$ in $\mathcal{H}_n$ is said to be convex if $f(\mathbb{P} \Delta(0;1))$ is convex. We denote the class of pluriharmonic convex mappings in $\mathbb{P} \Delta(0;1)$ by $\mathcal{K}_{\mathcal{H}_n}$. A function $f\in\mathcal{H}_n$ is said to be close-to-convex if $f(\mathbb{P} \Delta(0;1))$ is a close-to-convex domain. We denote the class of pluriharmonic close-to-convex mappings in $\mathbb{P} \Delta(0;1)$ by $\mathcal{C}_{\mathcal{H}_n}$. Finally, denote by ${\mathcal{S}^*_{\mathcal{H}_n}}^0$, $\mathcal{K}_{\mathcal{H}_n}^0$ and $\mathcal{C}_{\mathcal{H}_n}^0$, the
	subclasses of $\mathcal{S}^*_{\mathcal{H}_n}$, $\mathcal{K}_{\mathcal{H}_n}$ and $\mathcal{C}_{\mathcal{H}_n}$ with $\ol{{\bf \nabla}}f(0)=(0,0,\ldots,0)$ respectively.
\end{defi}
A pluriharmonic mapping $f=h+\overline{g}$ on the unit polydisk $\mathbb{P} \Delta(0;1)$ is said to be stably pluriharmonic univalent (respectively, stably pluriharmonic convex, stably pluriharmonic starlike with respect to the origin, or stably pluriharmonic close-to-convex) if all mappings $f_{\lambda}=h +\lambda \overline{g}$ with $|\lambda|=1$ are univalent (respectively, convex, starlike with respect to the origin, or close-to-convex) in $\mathbb{P} \Delta(0;1)$.
\vspace{1.2mm}

In 2013, Hernandez and Martin \cite{Hernandez-Martin-2013} proved the following result.
\begin{theoB}\cite[Proposition 2.1]{Hernandez-Martin-2013} A sense preserving harmonic mapping $f= h+\ol g$ on $\mathbb{D}$ is stable harmonic univalent if and only if the analytic function $F=h+g$ is stable analytic univalent on $\mathbb{D}$.
\end{theoB}

\begin{defi} For $M>0$ and $z\in \mathbb{P} \Delta(0;1)$, let
\begin{align*}
\mathcal{B}_{\mathcal{H}_n^0}(M)=\left\lbrace f=g+\ol h\in \mathcal{H}_n^0: \sum\limits_{l=1}^n\sum\limits_{j=1}^n\sum\limits_{k=1}^n\left(\left|z_l\frac{\partial^2 h(z)}{\partial z_j\partial z_k}\right|+\left|z_l\frac{\partial^2 g(z)}{\partial z_j\partial z_k}\right|\right)\leq M \right\rbrace.
\end{align*}
\end{defi}

We will show that the class $\mathcal{B}_{\mathcal{H}_n^0}(M)$ is closely related to the class
\begin{align*}
\mathcal{B}_n(M)=\left\lbrace \phi\in\mathcal{A}: \sum\limits_{l=1}^n\sum\limits_{j=1}^n\sum\limits_{k=1}^n\left|z_l\frac{\partial^2 \phi(z)}{\partial z_j\partial z_k}\right|\leq M\;\;\text{for}\;\;z\in \mathbb{P} \Delta(0;1)\right\rbrace.
\end{align*}

Let $\mathcal{A}$ denote the class of holomorphic functions $f$ on the unit polydisk $\mathbb{P} \Delta(0;1)$ normalized by
	\begin{align*}
		f(0)= 0 \quad \text{and} \quad \nabla f(0) :=\left(\frac{\partial f(0)}{\partial z_1}, \ldots, \frac{\partial f(0)}{\partial z_n}\right).= (1, 1, \ldots, 1).
	\end{align*}
	In the case where $n=1$, we denote $\mathcal{B}_1(M)$ as $\mathcal{B}(M)$. The class $\mathcal{B}(M)$ was investigated by Mocanu \cite{Mocanu-1992}, Ponnusamy and Singh \cite{Ponnusamy-Singh-2002}, and Singh \cite{Singh-2001}. Subsequently, in 2009, the region of variability for functions in $\mathcal{B}(M)$ was studied by Ponnusamy et al. \cite{Ponnusamy-Allu-Vuorinen-2009}.\vspace{1.2mm}
	
	In $2018$, Ghosh and Allu \cite{Ghosh-Allu-2018} established a result providing a one-to-one correspondence between the classes $\mathcal{B}_{\mathcal{H}_n^0}(M)$ and $\mathcal{B}_n(M)$ for the case $n=1$.
\begin{theoC}\cite[Theorem 2.1]{Ghosh-Allu-2018} A harmonic mapping $f=h+\ol g$ is in $\mathcal{B}_{\mathcal{H}^0}(M)$ if, and only if, $F_{\varepsilon}=h+\varepsilon g$ is in $\mathcal{B}(M)$ for each $\varepsilon$ ($|\varepsilon|=1$).
\end{theoC}

It is known from \cite{Ponnusamy-Allu-Vuorinen-2009} that functions in $\mathcal{B}(M)$ are univalent whenever $M \leq 1$. Consequently, in view of Theorem C, if $f=h+\overline{g}$ belongs to $\mathcal{B}_{\mathcal{H}^0}(M)$, then the functions $F_{\varepsilon}=h+\varepsilon g$ belong to $\mathcal{B}(M)$ for all $|\varepsilon|=1$, and are thus univalent for $M \leq 1$. This implies that the function $F=h+g$ is stably holomorphic univalent. Furthermore, by Theorem B, functions in the class $\mathcal{B}_{\mathcal{H}^0}(M)$ are stably harmonic univalent for $M \leq 1$, which directly implies their univalence.\vspace{1.2mm}

The following result due to Ghosh and Allu \cite{Ghosh-Allu-2018} provides sharp coefficient bounds for functions in $\mathcal{B}_{\mathcal{H}^0}(M)$.

\begin{theoD}\cite[Theorem 2.2]{Ghosh-Allu-2018} Let $f=h+\ol g\in \mathcal{B}_{\mathcal{H}^0}(M)$ and 
\begin{align}\label{THD}
h(z)=z+\sum\limits_{k=2}^{\infty} a_{k}z^k\quad \text{and}\quad g(z)=\sum\limits_{k=2}^{\infty}b_k z^k.
\end{align}
Then for $k\geq 2$,
\begin{enumerate}
\item[\emph{(i)}] $|a_k|\leq \dfrac{M}{k(k-1)}$,\vspace{2mm}
\item[\emph{(ii)}] $|b_k|\leq \dfrac{M}{k(k-1)}$.
\end{enumerate}
Both inequalities are sharp.
\end{theoD}
In the same paper, Ghosh and Allu \cite{Ghosh-Allu-2018} obtained the following result that gives a sufficient condition for functions to belong to $\mathcal{B}_{\mathcal{H}^0}(M)$.

\begin{theoE}\cite[Theorem 2.4]{Ghosh-Allu-2018} Let $f=h+\ol g\in \mathcal{H}^0$ and be given by (\ref{THD}) and
\begin{align*}
\sum\limits_{k=2}^{\infty} k(k-1)(|a_k|+|b_k|)\leq M.
\end{align*}

Then $f\in \mathcal{B}_{\mathcal{H}^0}(M)$.
\end{theoE}
The novelty of this research lies in the significant extension of the theory of univalent harmonic functions from the complex plane ($\mathbb{C}$) to higher-dimensional complex spaces ($\mathbb{C}^n$). While previous studies by authors such as Clunie, Sheil-Small, and Ghosh and Allu primarily focused on planar harmonic mappings ($n=1$), this work generalizes several classical theorems to convex Reinhardt domains in $\mathbb{C}^n$. A key contribution is the development of a multidimensional version of the Noshiro-Warschawski Theorem tailored for pluriharmonic settings. Furthermore, by establishing a structural decomposition ($f=h+\overline{g}$) in simply connected domains of several variables, the paper provides a rigorous new pathway to bridge the gap between local derivative properties and stable pluriharmonic univalence in polydisks. The sharp coefficient bounds and univalence criteria provided for the class $\mathcal{B}_{\mathcal{H}_n^0}(M)$ represent a major advancement in the study of geometric function theory for several complex variables.

\section{{\bf Main Results}}\label{Sec-2}
The classical Noshiro-Warschawski Theorem serves as a cornerstone in geometric function theory, providing a simple yet powerful sufficient condition for the univalence of analytic functions based on the positivity of their derivative's real part. While this result is well-established for holomorphic functions in the unit disk, its extension to pluriharmonic mappings in higher dimensions requires a careful treatment of the partial derivatives across complex variables. In what follows, we generalize this principle to the setting of several complex variables within convex Reinhardt domains. By considering the alignment of the Jacobian matrix and the properties of the underlying geometry, we establish a robust univalence criterion. We now formally state this multidimensional analogue, which serves as the foundation for our subsequent results in the class $\mathcal{B}_{\mathcal{H}_{n}^{0}}(M)$.
\begin{theo}\label{Th-2.1} Suppose $f(z)=h(z)+\ol{g(z)}$ is a pluriharmonic mapping in a convex Reinhardt domain $\Omega\subset \mathbb{C}^n$ such that
\begin{align}\label{Th1:1.0} {\rm Re}\left(e^{i\gamma}\frac{\partial h(z)}{\partial z_j}\right)>\left|\frac{\partial g(z)}{\partial z_j}\right|\; \mbox{for all}\; j=1, 2, \ldots, n,
\end{align}
for all $z\in\Omega$ and for some $\gamma\in\mathbb{R}$. Then $f$ is univalent in $\Omega$.
\end{theo}
\begin{proof}[\bf Proof of Theorem \ref{Th-2.1}] Let $z=(z_1,z_2,\ldots,z_n)\in \Omega$ and $w=(w_1,w_2,\ldots,w_n) \in \Omega$ such that $z \neq w$. Since $z\neq w$, for the sake of simplicity we may assume that $z_j\neq w_j$ for all $j=1,\ldots,n$. Let $w_j-z_j=r_je^{i\alpha_j}$, $j=1,\ldots,n$.
Since $\Omega$ is a Reinhardt domain, it follows that $\tilde z=\left(e^{-i\alpha_1}z_1,e^{-i\alpha_2}z_2,\ldots,e^{-i\alpha_n}z_n\right)\in \Omega$ and $\tilde w=\left(e^{-i\alpha_1}w_1,e^{-i\alpha_2}w_2,\ldots,e^{-i\alpha_n}w_n\right)\in \Omega$. Note that $\tilde z\neq \tilde w$.
Then, because $\Omega$ is a convex domain, we see that
$\xi(t)=(1 - t) \tilde z + t \tilde w \in \Omega$ for $0 \le t \le 1$. Let $\varphi: [0,1]\to \mathbb{C}$ be defined by 
\begin{align*}
\varphi(t)=f(\xi(t))=h(\xi(t))+\ol{g(\xi(t))}.
\end{align*}
It is easy to see that $\varphi(1)=f(\tilde w)$ and $\varphi(0)=f(\tilde z)$. Moreover, we have 
\begin{align*}
h'(\xi(t)))=\frac{\partial h(\xi(t))}{\partial t}=\sum\limits_{j=1}^n\frac{\partial h(\xi(t))}{\partial z_j}.(w_j-z_j)e^{-i\alpha_j}
\end{align*}
and 
\begin{align*}
\frac{d\; \ol{g(\xi(t)))}}{d t}=\ol{\frac{\partial g(\xi(t))}{\partial t}}=\ol{\sum\limits_{j=1}^n\frac{\partial g(\xi(t))}{\partial z_j}(w_j-z_j)e^{-i\alpha_j}}=\sum\limits_{j=1}^n\ol{\frac{\partial g(\xi_j(t))}{\partial z_j}}.\ol{(w_j-z_j)e^{-i\alpha_j}}.
\end{align*}
Consequently, we have
\begin{align*}
\frac{d \varphi(t)}{dt}=\frac{d f(\xi(t))}{dt}=\sum\limits_{j=1}^n\frac{\partial h(\xi(t))}{\partial z_j}.(w_j-z_j)e^{-i\alpha_j}+\sum\limits_{j=1}^n\ol{\frac{\partial g(\xi(t))}{\partial z_j}}.\ol{(w_j-z_j)e^{-i\alpha_j}}.
\end{align*}
We observe that \begin{align*} \varphi(1)-\varphi(0)=\int_{0}^1 \frac{d \varphi(t)}{dt}dt, \end{align*} and so we have \begin{align}\label{Th1:1.1} f(\tilde w)-f(\tilde z)=\int\limits_{0}^1\left( \sum\limits_{j=1}^n r_j \frac{\partial h(\xi(t))}{\partial z_j}+\sum\limits_{j=1}^n r_j \overline{\frac{\partial g(\xi(t))}{\partial z_j}}\right)dt. \end{align}
Using (\ref{Th1:1.0}) to (\ref{Th1:1.1}), we obtain
\begin{align*}
|f(\tilde w)-f(\tilde z)|>&\int\limits_{0}^1 {\rm Re} \left( \sum\limits_{j=1}^nr_je^{i\gamma}\frac{\partial h(\xi(t))}{\partial z_j}+\sum\limits_{j=1}^nr_je^{i\gamma}\ol{\frac{\partial g(\xi(t))}{\partial z_j}}\right) dt\\=&
\int\limits_{0}^1  \left( \sum\limits_{j=1}^nr_j{\rm Re}\left(e^{i\gamma}\frac{\partial h(\xi(t))}{\partial z_j}\right)+\sum\limits_{j=1}^nr_j{\rm Re}\left(e^{i\gamma}\ol{\frac{\partial g(\xi(t))}{\partial z_j}}\right)\right) dt
\\>&
\int\limits_{0}^1  \left(\sum\limits_{j=1}^nr_j\left|\frac{\partial g(\xi(t))}{\partial z_j}\right|-\sum\limits_{j=1}^nr_j\left|\ol{\frac{\partial g(\xi(t))}{\partial z_j}}\right|\right) dt\\=&
\int\limits_{0}^1   \sum\limits_{j=1}^n r_j\left(\left|\frac{\partial g(\xi(t))}{\partial z_j}\right|-\left|\frac{\partial g(\xi(t))}{\partial z_j}\right|\right) dt=0\nonumber
\end{align*}
which proves the univalence of $f$ in $\Omega$.
\end{proof}
By applying the methodology used in the proof of Theorem 2.1, we immediately obtain the following corollary, which characterizes the univalence of the mapping's holomorphic and co-holomorphic parts separately.
\begin{cor}\label{Co-1.1} Suppose $f(z)$ is holomorphic in a convex Reinhardt domain $\Omega\subset \mathbb{C}^n$ such that
\begin{align*} {\rm Re}\left(\frac{\partial f(z)}{\partial z_1}\right)>0\; \mbox{for all}\; j=1, 2, \ldots, n,
\end{align*}
for all $z\in\Omega$. Then $f$ is univalent in $\Omega$.
\end{cor}

When $\dim(\mathbb{C}^n)=1$ and $\Omega\subset \mathbb{C}$ is convex, Corollary \ref{Co-1.1} is known as the Noshiro-warschawski Theorem (see \cite[Theorem 2.16]{Duren-1983}).

\begin{rem} Since an open polydisk $\mathbb{P} \Delta(a;r)\subset \mathbb{C}^n$ is a convex Reinhardt domain, one can conclude that Theorem \ref{Th-2.1} and Corollary \ref{Co-1.1} both hold also in $\mathbb{P} \Delta(a;r)$.
\end{rem}
The relationship between the univalence of a harmonic mapping and its holomorphic counterpart is a central theme in geometric function theory. While this connection is well-understood in the plane, extending the concept of stable univalence to higher dimensions provides a deeper understanding of the structural properties of pluriharmonic mappings. The following theorem establishes this correspondence on the unit polydisk, showing that stable pluriharmonic univalence is equivalent to the stable univalence of the associated holomorphic function. We state the multidimensional version of Theorem B.
\begin{theo}\label{Th-2.2} A pluriharmonic mapping $f= h+\ol g$ on $\mathbb{P} \Delta(0;1)$ is stable pluriharmonic univalent if, and only if, the holomorphic function $F=h+g$ is stable holomorphic univalent on $\mathbb{P} \Delta(0;1)$.
\end{theo}
\begin{proof}[\bf Proof of Theorem \ref{Th-2.2}] Assume that $f_{\lambda}= h+\lambda \ol g$ is univalent for every $\lambda$ such that $\lambda|=1$ and suppose that $F= h+g$ is not univalent on $\mathbb{P} \Delta(0;1)$. Thus, there exist two different points $z, w\in \mathbb{P} \Delta(0;1)$ such that $F(z)=F(w)$.
\vspace{1.2mm}

First we suppose $h(z)=h(w)$. Then $F(z)=F(w)$ implies that $g(z)=g(w)$ and so $f_{\lambda}(z)=f_{\lambda}(w)$. This is a contradiction with our assumption.
\vspace{1.2mm}

Next we suppose $h(z)\neq h(w)$. Let $\theta=\arg \{h(z)-h(w)\}\in [0, 2\pi)$. Since $F(z)=F(w)$, we have $h(z)-h(w)=g(w)-g(z)$. Thus
\begin{align*}
e^{-i\theta}(h(z)-h(w))=e^{-i\theta}(g(w)-g(z))
\end{align*}
is a positive real number. Therefore, taking the conjugate on both sides of the last equality yields
\begin{align*}
	e^{-i\theta}(h(z)-h(w))=e^{i\theta}(\overline{g}(w)-\overline{g}(z)),
\end{align*}
which implies that $h(z)+\mu \overline{g}(z)=h(w)+\mu \overline{g}(w)$, \textit{i.e.,} $f_{\mu}(z)=f_{\mu}(w)$, where $\mu=e^{2i\theta}$. This again leads to a contradiction. Consequently, $F = h + g$ is univalent on the unit polydisk $\mathbb{P} \Delta(0;1)$, and the proof is complete.
\end{proof}
After establishing the general univalence criteria in Theorem 2.2, we now show that the specific class $\mathcal{B}_{\mathcal{H}_{n}^{0}}(M)$ possesses a one-to-one correspondence with its holomorphic counterpart on the unit polydisk. This result extends the planar findings of Ghosh and Allu to higher dimensions, providing a clear framework for analyzing pluriharmonic mappings through holomorphic functions. We state the multidimensional version of Theorem C.
\begin{theo}\label{Th-1.3} A pluriharmonic mapping $f=h+\ol g$ is in $\mathcal{B}_{\mathcal{H}_n^0}(M)$ if, and only if, $F_{\varepsilon}=h+\varepsilon g$ is in $\mathcal{B}_n(M)$ for each $\varepsilon\; (|\varepsilon|=1)$.
\end{theo}
\begin{proof} Let $f\in \mathcal{B}_{\mathcal{H}_n^0}(M)$. Then for each $\varepsilon\;(|\varepsilon|=1)$, we have
\begin{align*}
\sum\limits_{l=1}^n\sum\limits_{j=1}^n\sum\limits_{k=1}^n\left|z_l\frac{\partial^2 F_{\varepsilon}(z)}{\partial z_j\partial z_k}\right|
=&\sum\limits_{l=1}^n\sum\limits_{j=1}^n\sum\limits_{k=1}^n\left|z_l\frac{\partial^2 h(z)}{\partial z_j\partial z_k}+\varepsilon z_l\frac{\partial^2 g(z)}{\partial z_j\partial z_k}\right|\\ \leq&
\sum\limits_{l=1}^n\sum\limits_{j=1}^n\sum\limits_{k=1}^n\left|z_l\frac{\partial^2 h(z)}{\partial z_j\partial z_k}\right|+\sum\limits_{l=1}^n\sum\limits_{j=1}^n\sum\limits_{k=1}^n\left|z_l\frac{\partial^2 g(z)}{\partial z_j\partial z_k}\right|\\\leq& M
\end{align*}
for all $z\in \mathbb{P} \Delta(0;1)$ and so $F_{\varepsilon}\in \mathcal{B}_n(M)$. Conversely, let $F_{\varepsilon}=h+\varepsilon g\in \mathcal{B}_n(M)$ for each $\varepsilon\;(|\varepsilon|=1)$. Then we have
\begin{align*}
\sum\limits_{l=1}^n\sum\limits_{j=1}^n\sum\limits_{k=1}^n\left|z_l\frac{\partial^2 F_{\varepsilon}(z)}{\partial z_j\partial z_k}\right|
=&\sum\limits_{l=1}^n\sum\limits_{j=1}^n\sum\limits_{k=1}^n\left|z_l\frac{\partial^2 h(z)}{\partial z_j\partial z_k}+\varepsilon z_l\frac{\partial^2 g(z)}{\partial z_j\partial z_k}\right|\leq M
\end{align*}
for all $z\in \mathbb{P} \Delta(0;1)$. If we take $\varepsilon=\pm 1$, then from above inequality we have respectively
\begin{align}\label{Th3:1.1}
\left|z_l\frac{\partial^2 h(z)}{\partial z_j\partial z_k}\pm z_l\frac{\partial^2 g(z)}{\partial z_j\partial z_k}\right|^2\leq M_{l,j,k}^2
\end{align}
for all $z\in \mathbb{P} \Delta(0;1)$ and for all $j,k,l\in\{1,2,\ldots,n\}$, where 
\begin{align*}
\sum\limits_{l=1}^n\sum\limits_{j=1}^n\sum\limits_{k=1}^n M_{l,j,k}=M.
\end{align*}

But we know that $|z+w|^2+|z-w|^2=2\left(|z|^2+|w|^2\right)$ for all $z,w\in\mathbb{C}$. From (\ref{Th3:1.1}), we have 
\begin{align}\label{Th3:1.2}
2\left(\left|z_l\frac{\partial^2 h(z)}{\partial z_j\partial z_k}\right|^2+\left|z_l\frac{\partial^2 g(z)}{\partial z_j\partial z_k}\right|^2\right)\leq M_{l,j,k}^2
\end{align}
for all $z\in \mathbb{P} \Delta(0;1)$ and for all $j,k,l\in\{1,2,\ldots,n\}$. On the other hand, for all $z, w \in \mathbb{C}^n$, we have
\begin{align*}
\left(|z|+|w|\right)^2=2\left(|z|^2+|w|^2\right)-\left(|z|-|w|\right)^2\leq 2\left(|z|^2+|w|^2\right).
\end{align*}
Thus it follows from (\ref{Th3:1.2}) that
\begin{align*}
\left|z_l\frac{\partial^2 h(z)}{\partial z_j\partial z_k}\right|+\left|z_l\frac{\partial^2 g(z)}{\partial z_j\partial z_k}\right|\leq M_{l,j,k}
\end{align*}
for all $z\in \mathbb{P} \Delta(0;1)$ and for all $j,k,l\in\{1,2,\ldots,n\}$. A simple computation leads to
\begin{align*}
\sum\limits_{l=1}^n\sum\limits_{j=1}^n\sum\limits_{k=1}^n\left|z_l\frac{\partial^2 h(z)}{\partial z_j\partial z_k}\right|+\sum\limits_{l=1}^n\sum\limits_{j=1}^n\sum\limits_{k=1}^n\left|z_l\frac{\partial^2 g(z)}{\partial z_j\partial z_k}\right|\leq \sum\limits_{l=1}^n\sum\limits_{j=1}^n\sum\limits_{k=1}^n M_{l,j,k}=M
\end{align*}
and so $f\in \mathcal{B}_{\mathcal{H}_n^0}(M)$.
\end{proof}

The following lemma is contained in \cite[Lemma 6.1.28]{Graham-Kohr}.
\begin{lem}\label{Lem1} Let $f$ be holomorphic in the polydisk $\mathbb{P}\Delta(0;1_n)$ such that $|f(z)|\leq 1$ for all $z\in \mathbb{P}\Delta(0;1_n)$. Suppose $k(\geq 1)$ is the zero-multiplicity of $f$ at $0$. Then
\begin{align*}
|f(z)|\leq ||z||_{\infty}^k\;\text{ for all}\; z\in \mathbb{P}\Delta(0;1_n).
\end{align*}
\end{lem}
Extending the univalence properties known for the planar class $\mathcal{B}(M)$, we show that functions in the multidimensional class $\mathcal{B}_{\mathcal{H}_{n}^{0}}(M)$ are univalent when $M$ is suitably restricted.
\begin{theo}\label{Th-1.4} Suppose $f\in \mathcal{B}_n(M)$ and $M\leq 1$. Then $f$ is univalent.
\end{theo}
\begin{proof} Let $f\in \mathcal{B}_n(M)$. Then
\begin{align}\label{Th4:1}
\sum\limits_{l=1}^n\sum\limits_{j=1}^n\sum\limits_{k=1}^n\left|z_l\frac{\partial^2 f(z)}{\partial z_j\partial z_k}\right|\leq M
\end{align}
for all $z=(z_1,z_2,\ldots,z_n)\in \mathbb{P} \Delta(0;1)$. Set $z=(z_1,z_2,\ldots,z_n)\in \mathbb{P} \Delta(0;1)$ such that $z\neq 0$. Clearly, $||z||_{\infty}<1$. Because $\mathbb{P} \Delta(0;1)$ is a convex domain, we have
\begin{align*}
	\xi(t)=\frac{z}{||z||_{\infty}}t \in \mathbb{P} \Delta(0;1)\; \mbox{for}\;0 \le t \le ||z||_{\infty}.
\end{align*} Let $\varphi_j: [0,||z||_{\infty}]\to \mathbb{C}$ be defined by 
\begin{align*}
\varphi_k(t)=\frac{\partial f(\xi(t))}{\partial z_k}
\end{align*}
for $k=1,2,\ldots,n$. Then, we see that $\varphi_k(||z||_{\infty})=\frac{\partial f(z)}{\partial z_k}$ and $\varphi_k(0)=\frac{\partial f(0)}{\partial z_k}$ for $k=1,2,\ldots,n$. Note that
\begin{align*}
\varphi_k'(t)=\frac{\partial \varphi_k(t)}{\partial t}=\sum\limits_{j=1}^n\frac{\partial^2 f(\xi(t))}{\partial z_j \partial z_k}.\frac{z_j}{||z||_{\infty}}
\end{align*}
for $j=1,2,\ldots,n$. We see that
\begin{align*}
\varphi_j(||z||_{\infty})-\varphi_j(0)=\int\limits_{0}^{||z||_{\infty}} \frac{d \varphi_j(t)}{dt}dt.
\end{align*}
Since $\frac{\partial f(0)}{\partial z_k}=1$ for $k=1,2,\ldots,n$, we have
\begin{align}\label{Th4:1.1}
\frac{\partial f(z)}{\partial z_k}-1=\int\limits_{0}^{||z||_{\infty}}\left( \sum\limits_{j=1}^n\frac{\partial^2 f(\xi(t))}{\partial z_j \partial z_k}.\frac{z_j}{||z||_{\infty}}\right) dt.
\end{align}

Let 
\begin{align}\label{Th410}
\omega_{jk}(z)=z_j\frac{\partial^2 f(z)}{\partial z_j\partial z_k}/M
\end{align}
where $z=(z_1,z_2,\ldots,z_n)\in \mathbb{P} \Delta(0;1)$ and $j,k\in\{1,2,\ldots,n\}$. Clearly $\omega_{jk}(z)$ is a holomorphic function in $\mathbb{P} \Delta(0;1)$ such that $\omega_l(0)=0$ for $j.k\in\{1,2,\ldots,n\}$. Also from (\ref{Th4:1}), we deduce that $|\omega_{jk}(z)|\leq 1$ in $\mathbb{P} \Delta(0;1)$ for $j,k\in \{1,2,\ldots,n\}$. By Lemma \ref{Lem1}, we obtain
\begin{align}\label{Th411}
|\omega_{jk}(z)|\leq ||z||_{\infty}
\end{align}
for $j,k\in\{1,2,\ldots,n\}$. Consequently from (\ref{Th410}) and (\ref{Th411}), we have
\begin{align}\label{Th413}
\sum\limits_{j=1}^n\left|z_j\frac{\partial^2 f(z)}{\partial z_j \partial z_k}\right|\leq nM||z||_{\infty}.
\end{align}
In view of (\ref{Th413}) and (\ref{Th4:1.1}), we obtain
\begin{align*}
\left|\left|\frac{\partial f(z)}{\partial z_k}\right|-1\right|\leq \left|\frac{\partial f(z)}{\partial z_k}-1\right|\leq\int\limits_{0}^{||z||_{\infty}}\left( \sum\limits_{j=1}^n\left|\frac{\partial^2 f(\xi(t))}{\partial z_j \partial z_k}.\frac{z_j}{||z||_{\infty}}\right|\right) dt\leq M||z||_{\infty}<M
\end{align*}
which shows that
\begin{align*}
1-M<\left|\frac{\partial f(z)}{\partial z_j}\right|<1+M
\end{align*}
for $j=1,2,\ldots,n$. Consequently, $\frac{\partial f(z)}{\partial z_k} \neq 0$ for all $z \in \mathbb{P} \Delta(0;1) \setminus \{0\}$ and each $k \in \{1, \dots, n\}$. Given that the normalization at the origin implies $\frac{\partial f(0)}{\partial z_k} = 1$, it follows that the partial derivatives are non-vanishing throughout the entire polydisc $\mathbb{P} \Delta(0;1)$. Thus, by Proposition 1.1, we conclude that $f$ is univalent in $\mathbb{P} \Delta(0;1)$.
\end{proof}

Theorem \ref{Th-2.1} implies that functions in $\mathcal{B}_n(M)$ are univalent provided $M \leq 1$. Consequently, in view of Theorem \ref{Th-2.2}, for any $f = h + \overline{g}$ in $\mathcal{B}_{\mathcal{H}_n^0}(M)$, the functions $F_{\varepsilon} = h + \varepsilon g$ ($|\varepsilon| = 1$) belong to $\mathcal{B}_n(M)$ and are univalent for $M \leq 1$. This ensures that $F = h + g$ is stably holomorphic univalent. By Theorem \ref{Th-2.2}, functions in the class $\mathcal{B}_{\mathcal{H}_n^0}(M)$ are stably pluriharmonic univalent when $M \leq 1$, and hence, they are univalent.
\vspace{1.2mm}

In the following theorem, we establish sharp coefficient estimates for functions in the class $\mathcal{B}_{\mathcal{H}_n^0}(M)$, thereby providing a multidimensional generalization of the planar results established by Ghosh and Allu in Theorem D.
\begin{theo}\label{Th-1.5} Let $f=h+\ol g\in \mathcal{B}_{\mathcal{H}_n^0}(M)$ and and be given by (\ref{Eq 1.7}). Then for any multi-index $\alpha=(\alpha_1,\alpha_2,\ldots,\alpha_n)$ such that $|\alpha|=m\geq 2$, we have
\begin{enumerate}
\item[\emph{(i)}] $\sum\limits_{|\alpha|=m}|a_{\alpha}|\leq \dfrac{\binom{m+n-1}{n-1}M}{nm(m-1)}$,\vspace{1.2mm}
\item[\emph{(ii)}] $\sum\limits_{|\alpha|=m}|b_{\alpha}|\leq \dfrac{\binom{m+n-1}{n-1}M}{nm(m-1)}$.
\end{enumerate}
Both inequalities are sharp.
\end{theo}
\begin{proof} Let $f=h+\ol g\in \mathcal{B}_{\mathcal{H}_n^0}(M)$. Then
\begin{align}\label{Th5:1.1}
\sum\limits_{l=1}^n\sum\limits_{j=1}^n\sum\limits_{k=1}^n\left|z_l\frac{\partial^2 h(z)}{\partial z_j\partial z_k}\right|\leq M-\sum\limits_{l=1}^n\sum\limits_{j=1}^n\sum\limits_{k=1}^n\left|z_l\frac{\partial^2 g(z)}{\partial z_j\partial z_k}\right|
\end{align}
for all $z=(z_1,z_2,\ldots,z_n)\in \mathbb{P} \Delta(0;1)$. It follows from (\ref{Eq 1.7}) that
\begin{align}\label{Th5:1.2}
h(z)=\sum\limits_{k=1}^n z_j+\sum\limits_{m=2}^{\infty}P_m(z)\quad {and} \quad g(z)=\sum\limits_{k=2}^{\infty} Q_m(z)
\end{align}
for all $z\in \mathbb{P} \Delta(0;1)$, where
\begin{align}\label{Th5:1.3}
P_{m}(z)=\sum\limits_{|\alpha|=m} a_{\alpha} z^{\alpha}\quad \text{and}\quad Q_m(z)=\sum\limits_{|\alpha|=m} b_{\alpha} z^{\alpha}
\end{align}
are homogeneous polynomials of degree $m\geq 2$ in $z\in \mathbb{P} \Delta(0;1)$. We know that the maximum number of terms in $\sum\limits_{|\alpha|=m}$ is $\binom{|\alpha|+n-1}{n-1}$. For the sake of simplicity, we may assume that $j\leq k\leq l$. A simple computation shows that
\begin{align}\label{Th5:1.4}
z_l\frac{\partial^2 P_m(z)}{\partial z_j\partial z_k}=\sum\limits_{|\alpha|=m}\alpha_{jk} a_{\alpha}z_1^{\alpha_1}\ldots z_{j-1}^{\alpha_{j-1}}z_{j}^{\alpha_j^*}z_{j+1}^{\alpha_{j+1}}\ldots z_{k-1}^{\alpha_{k-1}}z_{k}^{\alpha_k^{**}}z_{k+1}^{\alpha_{k+1}}\ldots z_l^{\alpha_l+1}\ldots z_n^{\alpha_n}
\end{align}
is a homogeneous polynomial of degree $m-1$ in $z_1,z_2,\ldots,z_n$, where
\begin{align*}
\alpha_{jk}=
\begin{cases}
\alpha_j(\alpha_j-1), & \text{if}\; j=k\;\text{and}\;\alpha_j=|\alpha|,\\[2ex]
\alpha_j(\alpha_j-1), & \text{if}\; j=k\;\text{and}\;2\leq \alpha_j<|\alpha|,\\[2ex]
\alpha_j\alpha_k,& \text{if}\; j\neq k
\end{cases}
\;\;,\quad 
\alpha^*_j=
\begin{cases}
\alpha_j-2,& \text{if}\; j=k,\\[2ex]
\alpha_j-1,& \text{if}\; j\neq k
\end{cases}
\end{align*}
and
\begin{align*}
\alpha^{**}_k=
\begin{cases}
\alpha_k,& \text{if}\; j=k,\\[2ex]
\alpha_k-1,& \text{if}\; j\neq k.
\end{cases}
\end{align*}
Combining (\ref{Th5:1.2}) through (\ref{Th5:1.4}), we obtain
\begin{align}\label{Th5:1.5}
z_l\frac{\partial^2 h(z)}{\partial z_j\partial z_k}&=
\sum\limits_{m=2}^{\infty} z_l\frac{\partial^2 P_m(z)}{\partial z_j\partial z_k}\\=&\sum\limits_{m=2}^{\infty}\sum\limits_{|\alpha|=m}\alpha_{jk} a_{\alpha}z_1^{\alpha_1}\ldots z_{j-1}^{\alpha_{j-1}}z_{j}^{\alpha_j^*}z_{j+1}^{\alpha_{j+1}}\ldots z_{k-1}^{\alpha_{k-1}}z_{k}^{\alpha_k^{**}}z_{k+1}^{\alpha_{k+1}}\ldots z_l^{\alpha_l+1}\ldots z_n^{\alpha_n}.\nonumber
\end{align}
By applying Cauchy's integral formula to $z_l \frac{\partial^2 h(z)}{\partial z_j \partial z_k}$, it follows from \eqref{Th5:1.5} that
\begin{align}\label{Th5:1.6}
&(2\pi i)^n \alpha_{jk}a_{\alpha}\\=&\int\limits_{|z_1|=r_1}\ldots \int\limits_{|z_n|=r_n}\frac{z_l \frac{\partial^2 h(z)}{\partial z_j\partial z_k}\;dz_1\;d z_2\ldots d z_n}{z_1^{\alpha_1+1}\ldots z_{j-1}^{\alpha_{j-1}+1}z_{j}^{\alpha_j}z_{j+1}^{\alpha_{j+1}+1}\ldots z_{k-1}^{\alpha_{k-1}+1}z_{k}^{\alpha_k}z_{k+1}^{\alpha_{k+1}+1}\ldots z_l^{\alpha_l+2}\ldots z_n^{\alpha_n}+1},\nonumber
\end{align}
where $0<r_j<1$ for $j=1,2,\ldots,n$. By setting $z_j = r_j e^{i\theta_j}$ for $0 \leq \theta_j \leq 2\pi$ and $j = 1, \dots, n$, we can apply the boundary condition established in (2.6). Consequently, we obtain
\begin{align}\label{Th5:1.7}
&(2\pi)^n\sum\limits_{|\alpha|=m}\alpha_{jk}|a_{\alpha}|\\\leq &\binom{|\alpha|+n-1}{n-1}\int\limits_{0}^{2\pi}\ldots \int\limits_0^{2\pi}\frac{\left|r_le^{i\theta_l} \frac{\partial^2 h(z)}{\partial z_j\partial z_k}\right|\;d\theta_1\;d \theta_2\ldots d \theta_n}{r_1^{\alpha_1}\ldots r_{j-1}^{\alpha_{j-1}}r_{j}^{\alpha_j-1}r_{j+1}^{\alpha_{j+1}}\ldots r_{k-1}^{\alpha_{k-1}}r_{k}^{\alpha_k-1}r_{k+1}^{\alpha_{k+1}}\ldots r_l^{\alpha_l+1}\ldots r_n^{\alpha_n}}.\nonumber
\end{align}
Setting $r_j = r$ for each $j = 1, 2, \dots, n$ in inequality (2.7) yields the simplified form
\begin{align}\label{Th5:1.8}
(2\pi)^nr^m \sum\limits_{l=1}^n\sum\limits_{|\alpha|=m}\sum\limits_{j=1}^n\sum\limits_{k=1}^n \alpha_{jk}|a_{\alpha}|&=
(2\pi)^nr^m \sum\limits_{|\alpha|=m}\sum\limits_{l=1}^n\sum\limits_{j=1}^n\sum\limits_{k=1}^n \alpha_{jk}|a_{\alpha}|\\\leq &\binom{|\alpha|+n-1}{n-1}\int\limits_{0}^{2\pi}\ldots \int\limits_0^{2\pi}\sum\limits_{l=1}^n\sum\limits_{j=1}^n\sum\limits_{k=1}^n\left|r_le^{i\theta_l}\frac{\partial^2 h(z)}{\partial z_j\partial z_k}\right|\;d\theta_1\;d \theta_2\ldots d \theta_n.\nonumber
\end{align}
However, we observe that if we set $\alpha_j = |\alpha| = m$ for all $j \in \{1, 2, \dots, n\}$, then
\begin{align*}
\sum\limits_{|\alpha|=m}\sum\limits_{j=1}^n\sum\limits_{k=1}^n\alpha_{jk}|a_{\alpha}|=\sum\limits_{|\alpha|=m}\sum\limits_{j=1}^n \alpha_j(\alpha_j-1)|a_{\alpha}|=m(m-1)\sum\limits_{\substack{j=1\\\alpha_j=m}}^n|a_{\alpha}|
\end{align*}
and if $\alpha_j<|\alpha|$ for $j=1,2,\ldots,n$, then we have
\begin{align*}
\sum\limits_{|\alpha|=m}\sum\limits_{j=1}^n\sum\limits_{k=1}^n\alpha_{jk}|a_{\alpha}|=&
\sum\limits_{|\alpha|=m}\left(\sum\limits_{j=1}^n \alpha_j(\alpha_j-1)+2\sum\limits_{\substack{j,k=1\\j\neq k}}^n \alpha_j\alpha_k\right)|a_{\alpha}|\\=&m(m-1)\sum\limits_{\substack{|(\alpha_1,\alpha_2,\ldots,\alpha_n)|=m\\\alpha_j<m}}|a_{\alpha}|.
\end{align*}
Thus, we see that
\begin{align}\label{Th5:1.8a}
\sum\limits_{|\alpha|=m}\sum\limits_{j=1}^n\sum\limits_{k=1}^n\alpha_{jk}|a_{\alpha}|=m(m-1)\sum\limits_{|\alpha|=m}|a_{\alpha}|.
\end{align}
For any multi-index $\nu=(\nu_1,\nu_2,\ldots,\nu_n)$, we have
\begin{align}\label{Th5:1.9}
\int\limits_{0}^{2\pi}\ldots \int\limits_0^{2\pi} (e^{i\theta_1})^{\nu_1}\ldots (e^{i\theta_n})^{\nu_n}d\theta_1 \ldots d\theta_n=
\begin{cases}
0,& \nu\neq (0,0,\ldots,0),\\[2ex]
(2\pi)^n,& \nu=(0,0,\ldots,0).
\end{cases}
\end{align}
Moreover, we see that
\begin{align*}
&\sum\limits_{l=1}^n\sum\limits_{j=1}^n\sum\limits_{k=1}^nz_l\frac{\partial^2 g(z)}{\partial z_j\partial z_k}\\=&
\sum\limits_{m=2}^{\infty}\sum\limits_{|\alpha|=m}\sum\limits_{l=1}^n\sum\limits_{j=1}^n\sum\limits_{k=1}^n\alpha_{jk} b_{\alpha}z_1^{\alpha_1}\ldots z_{j-1}^{\alpha_{j-1}}z_{j}^{\alpha_j^*}z_{j+1}^{\alpha_{j+1}}\ldots z_{k-1}^{\alpha_{k-1}}z_{k}^{\alpha_k^{**}}z_{k+1}^{\alpha_{k+1}}\ldots z_l^{\alpha_l+1}\ldots z_n^{\alpha_n}.\nonumber
\end{align*}
Using (\ref{Th5:1.9}), we deduce that
\begin{align*}
\frac{1}{(2\pi)^n}\int\limits_{0}^{2\pi}\ldots \int\limits_0^{2\pi}\sum\limits_{l=1}^n\sum\limits_{j=1}^n\sum\limits_{k=1}^n r_le^{i\theta_l}\frac{\partial^2 g(z)}{\partial z_j\partial z_k}\;d\theta_1\;d \theta_2\ldots d \theta_n=0.
\end{align*}
In view of (\ref{Th5:1.1}), we obtain from (\ref{Th5:1.8}) that
\begin{align*}
&r^m nm(m-1)\sum\limits_{|\alpha|=m}|a_{\alpha}|\\\leq &\binom{|\alpha|+n-1}{n-1}\frac{1}{(2\pi)^n}\int\limits_{0}^{2\pi}\ldots \int\limits_0^{2\pi}\left(M-\sum\limits_{l=1}^n\sum\limits_{j=1}^n\sum\limits_{k=1}^n\left|r_le^{i\theta_l}\frac{\partial^2 g(z)}{\partial z_j\partial z_k}\right|\right)\;d\theta_1\;d \theta_2\ldots d \theta_n\nonumber\\\leq&
\binom{|\alpha|+n-1}{n-1}M-\left|\frac{1}{(2\pi)^n}\int\limits_{0}^{2\pi}\ldots \int\limits_0^{2\pi}\sum\limits_{l=1}^n\sum\limits_{j=1}^n\sum\limits_{k=1}^n r_le^{i\theta_l}\frac{\partial^2 g(z)}{\partial z_j\partial z_k}\;d\theta_1\;d \theta_2\ldots d \theta_n\right|\\=&
\binom{|\alpha|+n-1}{n-1}M.
\end{align*}

Letting $r\to 1^-$, we obtain the coefficient bound $(i)$.\vspace{1.2mm}

\noindent Similarly, a parallel reasoning establishes that
\begin{align*}
\sum\limits_{|\alpha|=m}|b_{\alpha}|\leq \frac{\binom{m+n-1}{n-1}M}{nm (m-1)}
\end{align*}
if $m\geq 2$. \vspace{1.2mm}

To show that the bounds in (i) and (ii) are sharp, we consider the functions
\begin{align*}
	f_1(z)=\sum\limits_{j=1}^n z_j + \sum\limits_{|\alpha|=m} a_{\alpha} z^{\alpha} \quad \text{and} \quad f_2(z)=\sum\limits_{j=1}^n z_j + \sum\limits_{|\alpha|=m} \overline{b_{\alpha} z^{\alpha}},
\end{align*}
where
\begin{align*}
	a_{\alpha}=\frac{M}{n^2m(m-1)} \quad \text{and} \quad b_{\alpha}=\frac{M}{n^2m(m-1)}
\end{align*}
for every multi-index $\alpha=(\alpha_1,\alpha_2,\ldots,\alpha_n)$ satisfying $|\alpha|=m$.\vspace{1.2mm}

It is easy to verify that $f_1, f_2 \in \mathcal{B}_{\mathcal{H}_n^0}(M)$. Furthermore, a direct calculation shows that
\begin{align*}
	\sum_{|\alpha|=m} |a_{\alpha}(f_1)| = \frac{\binom{m+n-1}{n-1}M}{n^2 m(m-1)} \quad \text{and} \quad \sum_{|\alpha|=m} |b_{\alpha}(f_2)| = \frac{\binom{m+n-1}{n-1}M}{n^2 m(m-1)},
\end{align*}
which establishes the sharpness of the bounds and completes the proof.
\end{proof}
The following result provides a sufficient condition for a pluriharmonic mapping to belong to the class $\mathcal{B}_{\mathcal{H}_n^0}(M)$, established through an inequality involving its Taylor coefficients.
\begin{theo}\label{Th-1.6} Let $f=h+\ol g\in\mathcal{H}^0_n$ and be given by (\ref{Eq 1.7}) and
\begin{align}\label{Th6:1.1}
\sum\limits_{m=2}^{\infty}\sum\limits_{|\alpha|=m}n m(m-1)(|a_{\alpha}|+|a_{\alpha}|)\leq M.
\end{align}
Then $f\in \mathcal{B}_{\mathcal{H}_n^0}(M)$.
\end{theo}
\begin{proof} Let $f=h+\ol g\in \mathcal{H}^0_n$. Since $z=(z_1,\ldots,z_n)\in \mathbb{P} \Delta(0;1)$, using (\ref{Th5:1.5}) and (\ref{Th5:1.8a}), we obtain
\begin{align}\label{Th6:1.2}
\sum\limits_{l=1}^n\sum\limits_{j=1}^n\sum\limits_{k=1}^n\left|z_l\frac{\partial^2 h(z)}{\partial z_j\partial z_k}\right|\leq& \sum\limits_{m=2}^{\infty}\sum\limits_{l=1}^n\sum\limits_{|\alpha|=m}\sum\limits_{j=1}^n\sum\limits_{k=1}^n\alpha_{jk}|a_{\alpha}|\\=&
\sum\limits_{m=2}^{\infty}\sum\limits_{l=1}^n \sum\limits_{|\alpha|=m}m(m-1)|a_{\alpha}|\nonumber\\=&
\sum\limits_{m=2}^{\infty}\sum\limits_{|\alpha|=m}nm(m-1)|a_{\alpha}|\nonumber.
\end{align}
Similarly, it can be establish that
\begin{align}\label{Th6:1.3}
\sum\limits_{l=1}^n\sum\limits_{j=1}^n\sum\limits_{k=1}^n\left|z_l\frac{\partial^2 g(z)}{\partial z_j\partial z_k}\right|\leq
\sum\limits_{m=2}^{\infty}\sum\limits_{|\alpha|=m}nm(m-1)|b_{\alpha}|.
\end{align}
From \eqref{Th6:1.1} and \eqref{Th6:1.3}, we obtain
\begin{align*}
	\sum\limits_{l=1}^n\sum\limits_{j=1}^n\sum\limits_{k=1}^n\left|z_l\frac{\partial^2 h(z)}{\partial z_j\partial z_k}\right| &\leq M - \sum\limits_{m=2}^{\infty}\sum\limits_{|\alpha|=m} nm(m-1)|b_{\alpha}| \
	&\leq M - \sum\limits_{l=1}^n\sum\limits_{j=1}^n\sum\limits_{k=1}^n\left|z_l\frac{\partial^2 g(z)}{\partial z_j\partial z_k}\right|.
\end{align*}
Rearranging these inequalities, we find that
\begin{align*}
	\sum\limits_{l=1}^n\sum\limits_{j=1}^n\sum\limits_{k=1}^n \left( \left|z_l\frac{\partial^2 h(z)}{\partial z_j\partial z_k}\right| + \left|z_l\frac{\partial^2 g(z)}{\partial z_j\partial z_k}\right| \right) \leq M,
\end{align*}
which confirms that $f \in \mathcal{B}_{\mathcal{H}_n^0}(M)$.
\end{proof}
In the next result, we establish that the class $\mathcal{B}_{\mathcal{H}_n^0}(M)$ is closed under convex combinations.
\begin{theo} The class $\mathcal{B}_{\mathcal{H}_n^0}(M)$ is closed under convex combination.
\end{theo}
\begin{proof} Let $f_s=h_s+\ol g_s\in \mathcal{B}_{\mathcal{H}_n^0}(M)$ for $s=1,2,\ldots,m$ and $\sum_{s=1}^m t_s=1\;(0\leq t_s\leq 1)$.
The convex combination of the $f_s$'s can be written as
\begin{align*}
f(z)=\sum\limits_{s=1}^m t_sf_s=h(z)+\ol {g(z)},
\end{align*}
where $h(z)=\sum_{s=1}^m t_s h_s(z)$ and $g(z)=\sum_{s=1}^m t_s g_s(z)$. Then both $h$ and $g$ are holomorphic in $\mathbb{P} \Delta(0;1)$ with $h(0)=g(0)=0,$
\begin{align*}
	\;\left(\frac{\partial h(0)}{\partial z_1},\ldots,\frac{\partial h(0)}{\partial z_n}\right)=(1,\ldots,1)\;\text{and}\; \left(\frac{\partial g(0)}{\partial z_1},\ldots,\frac{\partial g(0)}{\partial z_n}\right)=(0,\ldots,0).
\end{align*}

A simple computation shows that
\begin{align*}
\sum\limits_{l=1}^n\sum\limits_{j=1}^n\sum\limits_{k=1}^n\left|z_l\frac{\partial^2 h(z)}{\partial z_j\partial z_k}\right|=&\sum\limits_{l=1}^n\sum\limits_{j=1}^n\sum\limits_{k=1}^n\left|z_l\sum\limits_{s=1}^mt_s\frac{\partial^2 h_s(z)}{\partial z_j\partial z_k}\right|\\\leq&
\sum\limits_{s=1}^mt_s\sum\limits_{l=1}^n\sum\limits_{j=1}^n\sum\limits_{k=1}^n\left|z_l\frac{\partial^2 h_s(z)}{\partial z_j\partial z_k}\right|\\ \leq&
\sum\limits_{s=1}^m t_s\left(M-\sum\limits_{l=1}^n\sum\limits_{j=1}^n\sum\limits_{k=1}^n\left|z_l\frac{\partial^2 g_s(z)}{\partial z_j\partial z_k}\right|\right)\\=&
M-\sum\limits_{l=1}^n\sum\limits_{j=1}^n\sum\limits_{k=1}^n\left|z_l\sum\limits_{s=1}^m t_s\frac{\partial^2 g_s(z)}{\partial z_j\partial z_k}\right|\\=&
M-\sum\limits_{l=1}^n\sum\limits_{j=1}^n\sum\limits_{k=1}^n\left|z_l\frac{\partial^2 g(z)}{\partial z_j\partial z_k}\right|.
\end{align*}
This shows that $f\in \mathcal{B}_{\mathcal{H}_n^0}(M)$ and completes the proof.
\end{proof}
Having established the convex structure of the class $\mathcal{B}_{\mathcal{H}_{n}^{0}}(M)$ in the previous result, we now seek to provide growth theorem for these mappings. The following theorem provides sharp upper and lower bounds for the modulus of $f(z)$, thereby characterizing the geometric distortion of pluriharmonic mappings within the unit polydisk $\mathbb{P}\Delta(0;1)$. These estimates generalize the classical growth results for planar harmonic mappings to the multidimensional setting
\begin{theo}Let $f=h+\ol g\in \mathcal{B}_{\mathcal{H}_n^0}(M)$ and and be given by (\ref{Eq 1.7}). Then
\begin{align*}
n||z||_{\infty}-\frac{Mn^2}{2}||z||_{\infty}^2\leq |f(z)|\leq n||z||_{\infty}+\frac{Mn^2}{2}||z||_{\infty}^2.
\end{align*}
The inequality is sharp.
\end{theo}
\begin{proof} Let $f=h+\ol g\in \mathcal{B}_{\mathcal{H}_n^0}(M)$. Then from Theorem 2.1, we see that the function $F_{\varepsilon}=h+g$
belongs to $\mathcal{B}_{n}(M)$ for each $\varepsilon (|\varepsilon|=1)$, and further
\begin{align}\label{Th8:1}
\sum\limits_{l=1}^n\sum\limits_{j=1}^n\sum\limits_{k=1}^n\left|z_l\frac{\partial^2 F_{\varepsilon}(z)}{\partial z_j\partial z_k}\right|\leq M
\end{align}
for all $z=(z_1,z_2,\ldots,z_n)\in \mathbb{P} \Delta(0;1)$. 
Let $z=(z_1,z_2,\ldots,z_n)\in \mathbb{P} \Delta(0;1)$ such that $z\neq 0$. It is evident that $||z||_{\infty}<1$. Because $\mathbb{P} \Delta(0;1)$ is a convex domain, 
$\xi(t)=\frac{z}{||z||_{\infty}}t \in \mathbb{P} \Delta(0;1)$ for $0 \le t \le ||z||_{\infty}$. Let $\varphi_j: [0,||z||_{\infty}]\to \mathbb{C}$ be defined by 
\begin{align*}
\varphi_k(t)=\frac{\partial F_{\varepsilon}(\xi(t))}{\partial z_k}
\end{align*}
for $k=1,2,\ldots,n$. Clearly, $\varphi_k(||z||_{\infty})=\frac{\partial F_{\varepsilon}(z)}{\partial z_k}$ and $\varphi_k(0)=\frac{\partial F_{\varepsilon}(0)}{\partial z_k}=1$ for $k=1,2,\ldots,n$. 
By a similar argument to that of Theorem \ref{Th-1.4}, we obtain
\begin{align}\label{Th8:1.1}
\left|\left|\frac{\partial F_{\varepsilon}(z)}{\partial z_k}\right|-1\right|\leq\left|\frac{\partial F_{\varepsilon}(z)}{\partial z_k}-1\right|\leq&\int\limits_{0}^{||z||_{\infty}}\left( \sum\limits_{j=1}^n\left|\frac{\partial^2 F_{\varepsilon}(\xi(t))}{\partial z_j \partial z_k}.\frac{z_j}{||z||_{\infty}}\right|\right) dt\\\leq& nM||z||_{\infty}\nonumber
\end{align}
for $k=1,2,\ldots,n$. Since \eqref{Th8:1.1} remains valid at $z=0$, we have
\begin{align}\label{Th8:1.2}
1-nM||z||_{\infty}\leq \left|\frac{\partial F_{\varepsilon}(z)}{\partial z_k}\right|=\left|\frac{\partial h(z)}{\partial z_k}+\varepsilon \frac{\partial g(z)}{\partial z_k}\right|\leq 1+nM||z||_{\infty}
\end{align}
for $k=1,2,\ldots,n$.
Since $\varepsilon (|\varepsilon|=1)$ is arbitrary, it follows from (\ref{Th8:1.2}) that
\begin{align}\label{Th8:1.3}
\left|\frac{\partial h(z)}{\partial z_k}\right|+\left|\frac{\partial g(z)}{\partial z_k}\right|\leq 1+nM||z||_{\infty}
\;\;\text{and}\;\;\left|\frac{\partial h(z)}{\partial z_j}\right|-\left|\frac{\partial g(z)}{\partial z_j}\right|\geq 1-nM||z||_{\infty}.
\end{align}
for $k=1,2,\ldots,n$.
Let $\varphi: \left[0,||z||_{\infty}\right]\to \mathbb{C}$ be defined by 
\begin{align*}
\varphi(t)=f(\xi(t))=h(\xi(t))+\ol{g(\xi(t))}.
\end{align*}
We observe that $\varphi(\|z\|_{\infty}) = f(z)$ and $\varphi(0) = f(0) = 0$. Note that
\begin{align*}
\frac{d \varphi(t)}{dt}=\frac{d f(\xi(t))}{dt}=\sum\limits_{j=1}^n\frac{\partial h(\xi(t))}{\partial z_j}.\frac{z_j}{||z||_{\infty}}+\sum\limits_{j=1}^n\ol{\frac{\partial g(\xi(t))}{\partial z_j}}.\frac{\ol z_j}{||z||_{\infty}}.
\end{align*}
\noindent Since 
\begin{align*}
	\varphi(1)-\varphi(0)=\int_{0}^1 \frac{d \varphi(t)}{dt}dt,
\end{align*} from (\ref{Th8:1.3}), we have
\begin{align*}
|f(z)|=&\left|\int\limits_{0}^{||z||_{\infty}}\left( \sum\limits_{k=1}^n\frac{\partial h(\xi(t))}{\partial z_k}.\frac{z_k}{||z||_{\infty}}+\sum\limits_{j=1}^n\ol{\frac{\partial g(\xi(t))}{\partial z_k}}.\frac{\ol z_k}{||z||_{\infty}}\right) dt\right|\\\leq&
\int\limits_{0}^{||z||_{\infty}}\sum\limits_{k=1}^n\left( \left|\frac{\partial h(\xi(t))}{\partial z_k}\right|+\left|\frac{\partial g(\xi(t))}{\partial z_k}\right|\right) \frac{|z_k|}{||z||_{\infty}} dt\\\leq&
\int\limits_{0}^{||z||_{\infty}}n(1+nMt)dt\\=&n||z||_{\infty}+\frac{Mn^2}{2}||z||_{\infty}^2.
\end{align*}
Moreover, using (\ref{Th8:1.3}), we obtain
\begin{align*}
|f(z)|=&\int\limits_{0}^{||z||_{\infty}}\left( \sum\limits_{k=1}^n\left|\frac{\partial h(\xi(t))}{\partial z_k}.\frac{z_k}{||z||_{\infty}}+\sum\limits_{j=1}^n\ol{\frac{\partial g(\xi(t))}{\partial z_k}}.\frac{\ol z_k}{||z||_{\infty}}\right|\right) dt\\\geq&
\int\limits_{0}^{||z||_{\infty}}\sum\limits_{k=1}^n\left( \left|\frac{\partial h(\xi(t))}{\partial z_k}\right|-\left|\frac{\partial g(\xi(t))}{\partial z_k}\right|\right) \frac{|z_k|}{||z||_{\infty}} dt\\\geq&
\int\limits_{0}^{||z||_{\infty}}n(1-nMt)dt\\=&n||z||_{\infty}-\frac{Mn^2}{2}||z||_{\infty}^2.
\end{align*}
Thus the desired inequalities are established.\vspace{2mm}

To show the inequalities are sharp, we consider the functions
\begin{align*}
f_1(z)=\sum\limits_{j=1}^n z_j+\frac{M}{2}\left(\sum\limits_{j=1}^n z_j\right)^2\quad \text{and} \quad f_2(z)=\sum\limits_{j=1}^n z_j-\frac{M}{2}\left(\sum\limits_{j=1}^n z_j\right)^2.
\end{align*}
For the point $z=(r,r,\ldots,r)$, where $r<1$, we find that
\begin{align*}
|f_1(z)|=nr+\frac{Mn^2}{2}r^2=n||z||_{\infty}+\frac{Mn^2}{2}||z||_{\infty}^2\quad \text{and} \quad |f_2(z)|=n||z||_{\infty}-\frac{Mn^2}{2}||z||_{\infty}^2,
\end{align*}
which shows that the inequalities are sharp. 
\end{proof}

\vspace{5mm}

\noindent\textbf{Conflict of interest:} The authors declare that there is no conflict  of interest regarding the publication of this paper.\vspace{1.2mm}

\noindent {\bf Funding:} Not Applicable.\vspace{1.2mm}

\noindent\textbf{Data availability statement:}  Data sharing not applicable to this article as no datasets were generated or analysed during the current study.\vspace{1.2mm}

\noindent {\bf Authors' contributions:} All the authors have equal contributions in preparation of the manuscript.

\end{document}